\documentclass[12pt]{amsart}
\usepackage[english]{babel}

\usepackage[latin1]{inputenc}
\usepackage{float}
\usepackage{latexsym}
\usepackage{amsmath}
\usepackage{amssymb}
\usepackage{graphicx}
\usepackage{amsthm}
\usepackage{amsfonts}

\usepackage[T1]{fontenc}
\newcommand{\assign}{:=}
\newcommand{\cdummy}{\cdot}
\newcommand{\nobracket}{}
\newcommand{\nocomma}{}
\newcommand{\noplus}{}
\newcommand{\nosymbol}{}
\newcommand{\plusassign}{+\!\!=}

\newcommand{\tmverbatim}[1]{{\ttfamily{#1}}}

\newtheorem{definition}{Definition}[section]
\newtheorem{corollary}[definition]{Corollary}
\newtheorem{example}[definition]{Example}
\newtheorem{proposition}[definition]{Proposition}
\newtheorem{remark}[definition]{Remark}
\newtheorem{theorem}[definition]{Theorem}

\newcommand{\xb}{\mathbf{x}}

\newcommand{\vb}{\mathbf{v}}
\newcommand{\wb}{\mathbf{w}}
\newtheorem{algo}{Algorithm}
\newcommand{\CC}{\mathbb{C}}
\newcommand{\Ac}{\mathcal{A}}

\newcommand{\RR}{\mathbb{R}}
\newcommand{\rank}{\mathrm{rank}}
\newcommand{\trace}{\mathrm{trace}}

\newcommand{\algorithm}[1]{\begin{algo}\ \\{#1}\end{algo}}

\begin{document}

\title{On the construction of general cubature formula by flat extensions}

\author{M. Abril Bucero}
\address{M.Abril Bucero, B. Mourrain: Inria Sophia Antipolis M\'editerran\'ee\\
2004 route des Lucioles, BP 93,\\
06902 Sophia Antipolis, France\\
\ \\}

\author{C. Bajaj}
\address{C. Bajaj: Department of Computer Science\\
The Institute of Computational Engineering and Sciences, \\
The University of Texas at Austin,\\
201 East 24th Street, POB 2.324A,\\
1 University Station, C0200 Austin, TX 78712-0027, USA}

\author{B. Mourrain}


\maketitle

\begin{abstract}
  We describe a new method to compute general cubature formulae. The
  problem is initially transformed into the computation of truncated Hankel operators
  with flat extensions. We then analyse the algebraic properties associated to flat
  extensions and show how to recover the cubature points and weights from the
  truncated Hankel operator. We next present an algorithm to test the flat extension
  property and to additionally compute the decomposition. To generate cubature
  formulae with a minimal number of points, we propose a new relaxation
  hierarchy of convex optimization problems  minimizing the nuclear norm
  of the Hankel operators. For a suitably high order of
  convex relaxation, the minimizer of the optimization problem
  corresponds to a cubature formula. Furthermore cubature formulae with a minimal
  number of points are associated to faces of the convex sets. We illustrate our
  method on some examples, and for each we obtain a new minimal
  cubature formula.
\end{abstract}

\section{Cubature formula}

\subsection{Statement of the problem}

Consider the integral for a continuous function $f$,
\[ I [f] = \int_{\Omega} w (\xb) f (\xb) d \xb \]
where $\Omega \subset \mathbb{R}^n$ and $w$ is a positive function on
$\Omega$.

We are looking for a cubature formula which has the form
\begin{equation}
  \label{eq:cubature} \langle \sigma |f \rangle = \sum_{j = 1}^r
  \hspace{0.17em} w_j  \hspace{0.17em} f (\zeta_j)
\end{equation}
where the points $\zeta_j \in \CC^n$ and the weights $w_j \in \RR$ are
independent of the function $f$. They are chosen so that
\[ \langle \sigma |f \rangle = I [f], \forall f \in V, \]
where $V$ is a finite dimensional vector space of functions. Usually, the
vector space $V$ is the vector space of polynomials of degree $\le d$, because
a well-behaved function $f$ can be approximated by a polynomial, so that $Q
[f]$ approximates the integral $I [f]$.

Given a cubature formula (\ref{eq:cubature}) for $I$, its algebraic degree is
the largest degree $d$ for which $I [f] = \langle \sigma |f \rangle$ for all
$f$ of degree $\le d$.

\subsection{Related works }

Prior approaches to the solution of cubature problem can be grouped into
roughly two classes. One, where the goal is to estimate the fewest weighted,
aka cubature points possible for satisfying a prescribed cubature rule of
fixed degree \cite{C1,Mo2,Mo3,My1,My6,P2}.
The other class, focusses on the determination and construction of cubature rules which would
yield the fewest cubature points possible \cite{MP,R,S,SX,Str1,Str3,X1,X3}.
In {\cite{R}}, for example, Radon
introduced a fundamental technique for constructing minimal cubature rules
where the cubature points are common zeros of multivariate orthogonal
polynomials. This fundamental technique has since been extended by many,
including for e.g. {\cite{P2,Str3,X3}} where notably, the paper {\cite{X3}}
use multivariate ideal theory, while {\cite{P2}} uses operator dilation
theory. In this paper, we propose another approach to the second class of
cubature solutions, namely, constructing a suitable finite dimensional Hankel
matrix and extracting the cubature points using sub operators of the Hankel
matrix {\cite{I}}. This approach is related to
{\cite{laurent-ams-2005,laurent-ima-2008,ML08}} and which in turn are based on
the methods of multivariate truncated moment matrices, their positivity and
extension properties {\cite{cf-hjm-1991,CF98,CF00}}.

Applicatons of such algorithms determining cubature rules and cubature points
over general domains occur in isogeometric modeling and finite element
analysis using generalized Barycentric finite elements
{\cite{GB2011,GRB2012,RGB2013,RGB2014}}. Additional applications
abound in numerical integration for low dimensional (6-100 dimensions)
convolution integrals that appear naturally in computational molecular
biology {\cite{BBBV2013,AB2014}}, as well in truly high dimensional
(tens of thousands of dimensions) integrals that occur in finance
{\cite{NT1996,CMO1997}}.

\subsection{Reformulation}

Let $R = \RR [\xb]$ be the ring of polynomials in the variables $\xb = (x_1,
\ldots, x_n)$ with coefficients in $\RR$. Let $R_d$ be the set of polynomials
of degree $\leqslant d$. The set of linear forms on $R$, that is, the set of
linear maps from $R$ to $\RR$ is denoted by $R^{\ast}$. The value of a linear
form $\Lambda \in R^{\ast}$ on a polynomial $p \in R$ is denoted by $\langle
\Lambda | \nobracket p \rangle$. The set  $R^{\ast}$ can be identified with the  ring of formal power series in new variables
$\mathbf{z}= (z_1, \ldots, z_n)$:
\begin{eqnarray*}
  R^{\ast} & \rightarrow & \RR [[\mathbf{z}]]\\
  \Lambda & \mapsto & \Lambda (\mathbf{z}) = \sum_{\alpha \in \mathbb{N}^n}
  \langle \Lambda | \nobracket \mathbf{x}^{\alpha} \rangle
  \mathbf{z}^{\alpha} .
\end{eqnarray*}
The coefficients $\langle \Lambda | \nobracket \mathbf{x}^{\alpha} \rangle$
of these series are called the {\em{moments}} of $\Lambda$.
The evaluation at a point $\zeta \in \RR^n$ is an element of $R$,  denoted by
$\mathbf{e}_{\zeta}$, and defined by $\mathbf{e}_{\zeta} : f \in R
\mapsto f (\zeta) \in \mathbb{R}$.
For any $p \in R$ and any $\Lambda \in R^{\ast}$, let $p \star \Lambda : q
\in R \mapsto \Lambda (p q)$.

\

\textbf{Cubature problem:} Let $V \subset R$ be the vector space of
polynomials and consider the linear form $\bar{I} \in V^{\ast}$ defined by
\begin{eqnarray*}
  \bar{I} : V & \rightarrow & \RR\\
  \vb & \mapsto & I [\vb]
\end{eqnarray*}
Computing a cubature formula for $I$ on $V$ then consists in finding a linear form
\[ \sigma = \sum_{i = 1}^r w_i \mathbf{e}_{\zeta_i} : f \mapsto \sum_{j =
   1}^r \hspace{0.17em} w_j f (\zeta_j) . \]
which coincides on $V$ with $\bar{I}$. In other words, given the linear form
$\bar{I}$ on $R_d$, we wish to find a linear form $\sigma = \sum_{i = 1}^r w_i
\mathbf{e}_{\zeta_i}$ which extends $\bar{I}$.

\section{Cubature formulae and Hankel operators}

To find such a linear form $\sigma \in R^{\ast}$, we exploit the properties of
its associated bilinear form $H_{\sigma} : (p, q) \in R \times R \rightarrow
\left\langle \sigma |p \hspace{0.17em} q \right\rangle$, or equivalently, the
associated Hankel operator:
\begin{eqnarray*}
  H_{\sigma} : R & \rightarrow & R^{\ast}\\
  p & \mapsto & p \star \sigma
\end{eqnarray*}
The kernel of $H_{\sigma}$ is $\ker H_{\sigma} = \left\{ p \in R \mid \forall
q \in R, \left\langle \sigma |p \hspace{0.17em} q \right\rangle = 0 \right\}$.
It is an ideal of $R$. Let $\mathcal{A}_{\sigma} = R / \ker H_{\sigma}$ be
the associated quotient ring.

The matrix of the bilinear form or the Hankel operator $H_{\sigma}$
associated to $\sigma$ in the monomial basis, and its dual are $(\langle \Lambda
| \nobracket \mathbf{x}^{\alpha + \beta} \rangle)_{\alpha, \beta \in
\mathbb{N}^n}$. If we restrict them to a space $V$ spanned by the monomial
basis $(\mathbf{x}^{\alpha})_{\alpha \in A}$ for some finite set $A \subset
\mathbb{N}^n$, we obtain a finite dimensional matrix
{\textbf{}}$[H_{\sigma}^{A, A}] = (\langle \Lambda | \nobracket
\mathbf{x}^{\alpha + \beta} \rangle)_{\alpha, \beta \in A}$, and which is a
Hankel matrix.. More generally, for any vector spaces $V, V' \subset R$, we
define the truncated bilinear form and Hankel operators: $H^{V, V'}_{\sigma} :
(v, v') \in V \times V' \mapsto \langle \sigma |p q \rangle \in \mathbb{R}$
and $H^{V, V'}_{\sigma} : v \in V \mapsto v \star \sigma \in V'^{\ast}$. If
$V$ (resp. $V'$) is spanned by a monomial set $\mathbf{x}^A$ for $A \subset
\mathbb{N}^n $(resp. $\mathbf{x}^B$ for $B \subset \mathbb{N}^n$), the
truncated bilinear form and truncated Hankel operator are also denoted by
$H_{\sigma}^{A, B}$. The associated  Hankel matrix in the monomial basis is then \
{\textbf{}}$[H_{\sigma}^{A, B}] = (\langle \Lambda | \nobracket
\mathbf{x}^{\alpha + \beta} \rangle)_{\alpha \in A, \beta \in B}$.

The main property that we will use to characterize a cubature formula is the
following (see \cite{LaurentSurvey09,lasserre:hal-00651759}):

\begin{proposition}
  A linear form $\sigma \in R^{\ast}$ can be decomposed as $\sigma = \sum_{i =
  1}^r w_i \hspace{0.17em} \mathbf{e}_{\zeta_i}$ with $w_i \in \CC \setminus
  \{0\}$, $\zeta_i \in \CC^n$ iff
  \begin{itemize}
    \item $H_{\sigma} : p \mapsto p \star \sigma$ is of rank $r$,

    \item $\ker H_{\sigma}$ is the ideal of polynomials vanishing at the
    points $\{ \zeta_1, \ldots, \zeta_r \}$.
  \end{itemize}
\end{proposition}

This shows that in order to find the points $\zeta_i$ of a cubature formula,
it is sufficient to compute the polynomials $p \in R$ such that $\forall q \in
R$, $\langle \sigma |p q \rangle = 0$ , and to determine their common zeroes.
In section \ref{sec:root}  we  describe a direct way to recover the points $\zeta_i$, and
the weights $\omega_i${\textbf{}} from suboperators of $H_{\sigma}$.

In the case of cubature formulae with real points and positive weights, we already
have the following stronger result (see {\cite{LaurentSurvey09,lasserre:hal-00651759}}):

\begin{proposition}
  \label{prop:2}Let {\em{}}$\sigma \in R^{\ast}$.
  \[ \sigma = \sum_{i = 1}^r w_i \hspace{0.17em} \mathbf{e}_{\zeta_i} \]
  with $w_i > 0$, $\zeta_i \in \RR^n$ iff $\rank H_{\sigma} = r$ and
  $H_{\sigma} \succcurlyeq 0$.
\end{proposition}

A linear form $\sigma = \sum_{i = 1}^r w_i \hspace{0.17em}
\mathbf{e}_{\zeta_i}$ with $w_i > 0$, $\zeta_i \in \RR^n$ is called a
$r$-atomic measure since it coincides with the weighted sum of the $r$ Dirac
measures at the points $\zeta_i$.

Therefore, the problem of constructing a cubature formula $\sigma$ for $I$
exact on $V \subset R$ can be reformulated as follows: Construct a linear form
$\sigma \in R^{\ast}$ such that
\begin{itemize}
  \item $\rank H_{\sigma} = r < \infty$ and $H_{\sigma} \succcurlyeq 0$.

  \item $\forall v \in V$, $I [v] = \langle \sigma |  \nobracket v \rangle$.
\end{itemize}
The rank $r$ of $H_{\sigma}${\textbf{}} is  given by the number of points of the
cubature formula, which is expected to be small or even minimal.

The following result states that a cubature formula with $\dim (V)$ points,
always exist.

\begin{theorem}
  {\cite{Tchakaloff57,BaTei06}} If a sequence
  $(\sigma_{\alpha})_{\alpha \in \mathbb{N}^n, | \alpha | \leqslant t}$ is
  the truncated moment sequence of a measure ${mu}$ (i.e.
  $\sigma_{\alpha} = \int \mathbf{x}^{\alpha} d{mu}$ for $| \alpha |
  \leqslant t$) then it can also be represented by an $r$-atomic measure: for
  $| \alpha | \leqslant t$, $\sigma_{\alpha} = \sum_{i = 1}^r w_i
  \hspace{0.17em} \xi_i^{\alpha} $ where $r \leqslant s_t$, $w_i > 0$,
  $\zeta_i \in {supp} ({mu})$.
\end{theorem}

This result can be generalized to any set of linearly independent polynomials
$v_1, \ldots, v_r \in R$ (see the proof in {\textbf{}}{\cite{BaTei06}} or
Theorem 5.9 in {\cite{LaurentSurvey09}}). We deduce that the
cubature problem always has a solution with $\dim (V)$ or less points. \

\begin{definition}
  Let $r_c (I)$ be the maximum rank of the bilinear form $H_I^{W, W'} : (w,
  w') \in W \times W' \mapsto I [w w']$ where $W, W' \subset V$ are such that
  $\forall w \in W, \forall w' \in W'$, $w \hspace{0.17em} w' \in V$. It is
  called the Catalecticant rank of $I$.
\end{definition}

\begin{proposition}
  \label{prop:minr}Any cubature formula for $I$ exact on $V$ involves at least
  $r_c (I)$ points.
\end{proposition}

\begin{proof}
  Suppose that $\sigma$ is a cubature formula for $I$ exact on $V$ with $r$
  points. Let $W, W' \subset V$ be vector spaces are such that $\forall w \in
  W, \forall w' \in W'$, $w \hspace{0.17em} w' \in V$. Since $H_I^{W, W'}$
  coincides with $H_{\sigma}^{W, W'}$, which is the restriction of the
  bilinear form $H_{\sigma}$ to $W \times W'$, we deduce that $r = \rank
  H_{\sigma} \geqslant \rank H_{\sigma}^{W, W'} = \rank  (H_I^{W, W'})$. Thus
  $r \geqslant r_c (I) .$
\end{proof}

\begin{corollary}
  Let $W \subset V$ such that $\forall w, w' \in W, w \hspace{0.17em} w' \in
  V$. Then any cubature formula of $I$ exact on $V$ involves at least $\dim
  (W)$ points.
\end{corollary}

\begin{proof}
  As we have $\forall p \in W$, $p^2 \in V$ so that $I (p^2) = 0$ implies $p
  = 0$. Therefore the quadratic form $H_I^{W, W} : (p, q) \in W \times W
  \rightarrow I [p q]$ is positive definite of rank $\dim (W)$. By Proposition
  \ref{prop:minr}, a cubature formula of $I$ exact on $V$ involves at least
  $r_c (I) \geqslant \dim (W)$ points.
\end{proof}

In particular, if $V = R_d$ any cubature formula of $I$ exact on $V$ involve
at least $\dim R_{\lfloor \frac{d}{2} \rfloor} = \binom{\lfloor \frac{d}{2}
\rfloor + n}{n}$ points.

In {\cite{Moller79}}, this lower bound is improved for cubature problems in
two variables.

\section{Flat extensions}

In order to reduce the extension problem to a finite-dimensional problem, we
consider hereafter only truncated Hankel operators. Given two subspaces $W, W'$ of
$R$ and a linear form $\sigma$ defined on $W \cdummy W'$ (i.e. $\sigma \in
\langle W \cdummy W' \rangle^{\ast}$), we define
\begin{eqnarray*}
  H_{\sigma}^{W, W'} : W \times W' & \rightarrow & \RR\\
  (w, w') & \mapsto & \langle \sigma |w w' \rangle .
\end{eqnarray*}
If $\wb$ (resp. $\wb'$) is a basis of $W$ (resp. $W'$), then we will also
denote $H_{\sigma}^{\wb, \wb'} \assign H_{\sigma}^{W, W'}$. The matrix of
$H_{\sigma}^{\wb, \wb'}$ in the basis \ $\wb = \{ w_1, \ldots, w_s \}
\nocomma, \wb' = \{ w'_1, \ldots, w'_{s'} \}$ is $[\langle \sigma |
\nobracket w_i w_j \rangle]_{1 \leqslant i \leqslant s, 1 \leqslant j
\leqslant s'}$.

\begin{definition}
  Let $W \subset V$, $W' \subset V'$ be subvector spaces of $R$
  and $\sigma \in \langle V \cdummy V' \rangle^{\ast}$. We say that
  $H_{\sigma}^{V, V'}$ is a \textit{flat extension} of $H_{\sigma}^{W, W'}$
  if $\rank H_{\sigma}^{V, V'} = \rank \hspace{0.17em} H_{\sigma}^{W,
  W'}$.
\end{definition}

A set $B$ of monomials of $R$ is \textit{connected to $1$} if it contains
$1$ and if for any $m \neq 1 \in B$, there exists $1 \le i \le n$ and $m' \in
B$ such that $m = x_i m'$.

As a quotient $R / \ker \hspace{0.17em} \tilde{Q}$ has a monomial basis
connected to $1$, so in the first step we take for $\wb$, $\wb'$, monomial sets that are
connected to $1$ .

For a set $B$ of monomials in $R$, let us define $B^+ = B \cup x_1 B \cup
\cdots \cup x_n, B$ and $\partial B = B^+ \setminus B$.

The next theorem gives a characterization of flat extensions for Hankel
operators defined on monomial sets connected to $1$. It is a generalized form
of the Curto-Fialkow theorem {\cite{CF00}}.

\begin{theorem}
  {\cite{ML08,BracCMT09:laa,BBCMMultiSym13}}.\label{thm:flat:extension} Let $B
  \subset C, B' \subset C'$ be sets of \ monomials connected to $1$ such that
  $| B | = | B' | = r$ and $C \cdot C'$ contains $B^+ \cdot B'^+$. If $\sigma
  \in \langle C \cdot C' \rangle^{\ast}$ is such that $\rank \hspace{0.17em}
  H_{\sigma}^{B, B'} = \rank \hspace{0.17em} H_{\sigma}^{C, C'} = r$, then
  $\hspace{0.17em} H_{\sigma}^{C, C'}$ has a unique flat extension
  $H_{\tilde{\sigma}}$ for some $\tilde{\sigma} \in R^{\ast}$. Moreover, we
  have $\ker H_{\tilde{\sigma}} = (\ker H_{\sigma}^{C, C'})$ and $R = \langle
  B \rangle \oplus \ker H_{\tilde{\sigma}} = \langle B' \rangle \oplus \ker
  H_{\tilde{\sigma}}$. In the case where $B' = B$, if $H_{\sigma}^{B, B}
  \succcurlyeq 0$, then $H_{\tilde{\sigma}} \succcurlyeq 0$.
\end{theorem}

Based on this theorem, in order to find a flat extension of $H_{\sigma}^{B, B'}$,  it suffices
 to construct an extension $H_{\sigma}^{B^+, B'^+}$ of the same rank
$r$.

\begin{corollary}
  Let $V \subset R$ be a finite dimensional vector space. If there exist a set
  $B$ of \ monomials connected to $1$ such that $V \subset \langle B^+ \cdot
  B^+ \rangle$ and $\sigma \in \langle B^+ \cdot B^+ \rangle^{\ast}$ such that
  $\forall v \in V, \langle \sigma |v \rangle = I [v]$ and $\rank
  \hspace{0.17em} H_{\sigma}^{B, B} = \rank \hspace{0.17em} H_{\sigma}^{B^+,
  B^+} = | B | = r$, then there exists $w_i > 0$, $\zeta_i \in \RR^n$ such
  that $\forall v \in V,$
  \[ I [v] = \sum_{i = 1}^r w_i \hspace{0.17em} v (\xi_i) . \]
\end{corollary}

This characterization leads to equations which are at most of degree $2$ in a
set of variables related to unknown moments and relation coefficients as
described by the following proposition:

\begin{proposition}
  \label{prop:dec}Let $B$ and $B'$ be two sets of monomials of $R {of}
  {size}$r, connected to $1$ and $\sigma$ be a linear form on $\langle
  B^{' +} \cdot B^+ \rangle$. Then, $H^{B^+, B'^+}_{\sigma}$ admits a flat
  extension $H_{\tilde{\sigma}}$ such that $H_{\tilde{\sigma}}$ is of rank $r$
  and $B$ (resp. B') a basis of $R / \ker H_{\tilde{\sigma}}$ iff
  \begin{equation}
    \label{dec:H} [H_{\sigma}^{B^+, B'^+}] = \left( \begin{array}{cc}
      \mathbb{Q} & \mathbb{M}'\\
      \mathbb{M}^t & \mathbb{N}
    \end{array} \right),
  \end{equation}
  with $\mathbb{Q} = [H_{\sigma}^{B, B'}]$, $\mathbb{M}' = [H_{\sigma}^{B, \partial
  B'}]$, $\mathbb{M}^t = [H_{\sigma}^{\partial B, B'}]$, $\mathbb{N} =
  [H_{\sigma}^{\partial B, \partial B'}]$ is such that $\mathbb{Q}$ is invertible
  and
  \begin{equation}
    \label{eq:W} \mathbb{M} = \mathbb{Q}^t \hspace{0.17em} \mathbb{P}, \mathbb{M}' =
    \mathbb{Q} \hspace{0.17em} \mathbb{P}', \mathbb{N} = \mathbb{P}^t \hspace{0.17em}
    \mathbb{Q} \hspace{0.17em} \mathbb{P}',
  \end{equation}
  for some matrices $\mathbb{P} \in \CC^{B \times \partial B'}$, $\mathbb{P}' \in
  \CC^{B' \times \partial B}$.
\end{proposition}

\begin{proof}
  If we have $\mathbb{M} = \mathbb{Q}^t \hspace{0.17em} \mathbb{P}, \mathbb{M}' = \mathbb{Q}
  \hspace{0.17em} \mathbb{P}', \mathbb{N} = \mathbb{P}^t \hspace{0.17em} \mathbb{Q}
  \hspace{0.17em} \mathbb{P}'$, then
  \[ [H_{\sigma}^{B^+, B'^+}] = \left( \begin{array}{cc}
       \mathbb{Q} & \mathbb{Q} \hspace{0.17em} \mathbb{P}'\\
       \mathbb{P}^t \mathbb{Q} & \mathbb{P}^t \hspace{0.17em} \mathbb{Q} \hspace{0.17em}
       \mathbb{P}'
     \end{array} \right) \]
  has clearly the same rank as $\mathbb{Q}$=$[H_{\sigma}^{B, B'}]$. According to
  Theorem \ref{thm:flat:extension}, $H_{\sigma}^{B^+, B'^+}$ admits a flat
  extension $H_{\tilde{\sigma}}$ with $\tilde{\sigma} \in R^{\ast}$ such that
  $B$ and $B'$ are bases of $\mathcal{A}_{\tilde{\sigma}} = R / \ker
  H_{\tilde{\sigma}}$.

  Conversely, if $H_{\tilde{\sigma}}$ is a flat extension of \
  $H_{\sigma}^{B^+, B'^+}$ with $B$ and $B'$ bases of
  $\mathcal{A}_{\tilde{\sigma}} = R / \ker H_{\tilde{\sigma}}$, then
  $[H_{\tilde{\sigma}}^{B, B'}] = [H_{\sigma}^{B, B'}] = \mathbb{Q}$ is
  invertible and of size $r = | B | = | B' |$. As $H_{\tilde{\sigma}}$ is of
  rank $r$, $H_{\sigma}^{B^+, B'^+}$ is also of rank $r$. Thus, there exists
  $\mathbb{P}' \in \CC^{B' \times \partial B}$ ($\mathbb{P}' = \mathbb{Q}^{- 1}
  \mathbb{M}'$) such that $\mathbb{M}' = \mathbb{Q} \hspace{0.17em} \mathbb{P}'$.
  Similarly, there exists $\mathbb{P} \in \CC^{B \times \partial B'}$ such that
  $\mathbb{M} = \mathbb{Q}^t \hspace{0.17em} \mathbb{P}$. Thus, the kernel of
  $[H_{\sigma}^{B^+, B'^+}]$ (resp. $[H_{\sigma}^{B^{' +},
  B^+}]$=$[H_{\sigma}^{B^+, B'^+}]^t$) is the image of \ $\left(
  \begin{array}{c}
    - \mathbb{P}'\\
    \mathbb{I}
  \end{array} \right)$ (resp. $\left( \begin{array}{c}
    - \mathbb{P}\\
    \mathbb{I}
  \end{array} \right)$).
We deduce that
 $\mathbb{N} = \mathbb{M}^t \hspace{0.17em} \mathbb{P}' = \mathbb{P}^t \hspace{0.17em} \mathbb{Q} \hspace{0.17em} \mathbb{P}'$.
\end{proof}

\begin{remark}
  A basis of the kernel of $H_{\sigma}^{B^+, B'^+}$ is given by the columns of
  $\left( \begin{array}{c}
    - \mathbb{P}'\\
    \mathbb{I}
  \end{array} \right)$, which represent polynomials of the form
  \[ p_{\alpha} = \xb^{\alpha} - \sum_{\beta \in B} p_{\alpha, \beta}
     \hspace{0.17em} \xb^{\beta} \]
  for $\alpha \in \partial B$. These polynomials are border relations which
  project the monomials $\xb^{\alpha}$ of $\partial B$ on the vector space
  spanned by the monomials $B$, modulo $\ker H_{\sigma}^{B^+, B'^+}$. It is
  proved in {\cite{BracCMT09:laa}} that they form a border basis
  of the ideal $\ker H_{\tilde{\sigma}}$ when $H_{\sigma}^{B^+, B'^+}$ is a
  flat extension and $H_{\sigma}^{B, B'}$ is invertible.
\end{remark}

\begin{remark}
  Let $A \subset \mathbb{N}^n$ be a set of monomials such that $\langle
  \sigma |  \nobracket \mathbf{x}^{\alpha} \rangle = I
  [\mathbf{x}^{\alpha}]$. Considering the entries of $\mathbb{P}, \mathbb{P'}$ and
  the entries $\sigma_{\alpha}$ of $\mathbb{Q}$ with $\alpha \not\in A$ as
  variables, the constraints (\ref{eq:W}) are multilinear equations in these
  variables of total degree at most $3$ if $\mathbb{Q}$ contains unknown entries
  and $2$ otherwise.
\end{remark}

\begin{example}
  We consider here $V = R_{2 k}$ for $k > 0$. By Proposition \ref{prop:dec},
  any cubature formula for $I$ exact on $V$ has at least $r_k \assign \dim
  R_k$ points. Let us take $B$ to be all the monomials of degree $\le k$ so that
  $B^+$ is the set of monomials of degree $\le k + 1$. If a cubature formula \
  for $I$ is exact on $R_{2 k}$ and has $r_k$ points, then $H_{\sigma}^{B^+, B^+}$ is
  a flat extension of $H_{\sigma}^{B, B}$ of{\textbf{}} rank $r_k$. Consider
  a decomposition of \ $H_{\sigma}^{B^+, B^+}$ as in (\ref{dec:H}). By
  Proposition \ref{prop:dec}, we have the relations
  \begin{equation}
    \mathbb{M} = \mathbb{Q} \hspace{0.17em} \mathbb{P}, \mathbb{N} = \mathbb{P}^t
    \hspace{0.17em} \mathbb{Q} \hspace{0.17em} \mathbb{P}, \label{eq:4}
  \end{equation}
  where
  \begin{itemize}
    \item $\mathbb{Q} = (I [\xb^{\beta + \beta'}])_{\beta, \beta' \in B}$,

    \item $\mathbb{M} = (\left\langle \sigma |  \nobracket \xb^{\beta + \beta'}
    \right\rangle)_{\beta \in B, \beta' \in \partial B}$ with $\left\langle
    \sigma |  \nobracket \xb^{\beta + \beta'} \right\rangle = I [\xb^{\beta +
    \beta'}]$ when $| \beta + \beta' | \le 2 \hspace{0.17em} k$,

    \item $\mathbb{N}= (\left\langle \sigma |  \nobracket \xb^{\beta +
    \beta'} \right\rangle)_{\beta, \beta' \in \partial B}$

    \item $\mathbb{P} = (p_{\beta, \alpha})_{\beta \in B, \alpha \in \partial
    B}$.
  \end{itemize}
  The equations (\ref{eq:4}) are quadratic in the variables $\mathbb{P}$ and
  linear in the variables in $\mathbb{M}$. Solving these equations yields a flat
  extension $H_{\sigma}^{B^+, B^+}$ of $H_{\sigma}^{B, B}$. As $H_{\sigma}^{B,
  B} \succcurlyeq 0$, any real solution of this system of equations
  corresponds to a cubature for I on exact $R_{2 k}$ of the form $\sigma =
  \sum_{i = 1}^{r_k} w_i \hspace{0.17em} \mathbf{e}_{\zeta_i}$ with $w_i >
  0$, $\zeta_i \in \RR^n$.

  We illustrate the approach with $R = \RR [x_1, x_2]$, $V = R_4$, $B = \{1,
  x_1, x_2, x_1^2, x_1 x_2, x_2^2 \}$, $B^+ = \{1, x_1, x_2, x_1^2, x_1 x_2,
  x_2^2, x_1^3, x_1^2 x_2, x_1 x_2^2, x_2^3 \}$. Let
  {\small%
    \begin{eqnarray*}
      \sigma & = & 8 + 17 \hspace{0.17em} z_2 - 4 \hspace{0.17em} z_1
                   + 15 \hspace{0.17em} {z_2^2}
                   + 14 \hspace{0.17em} z_1 z_2
                   - 16 \hspace{0.17em} {z_1^2}
                   + 47 \hspace{0.17em} {z_2^3}
                 - 6 \hspace{0.17em} {z_1 z_2^2}
                 + 34 \hspace{0.17em} {z_1^2 z_2}
                 - 52 \hspace{0.17em} {z_1^3}\\
&&                 + 51 \hspace{0.17em}{z_2^4}
                 + 38 \hspace{0.17em} z_1  {z_2^3}
                 - 18  \hspace{0.17em} {z_1^2 z_2^2}
                 + 86 \hspace{0.17em} {z_1^3 z_2}
                 - 160 \hspace{0.17em} {z_1^4}
    \end{eqnarray*}
    }%
  be the series truncated in degree $4$, corresponding to the first moments
  (not necessarily given by an integral).
 {\small \[ H_{\sigma}^{B^+, B^+} = \left[ \begin{array}{cccccccccc}
       8 & - 4 & 17 & - 16 & 14 & 15 & - 52 & 34 & - 6 & 47\\
       - 4 & - 16 & 14 & - 52 & 34 & - 6 & - 160 & 86 & - 18 & 38\\
       17 & 14 & 15 & 34 & - 6 & 47 & 86 & - 18 & 38 & 51\\
       - 16 & - 52 & 34 & - 160 & 86 & - 18 & \sigma_1 & \sigma_2 & \sigma_3 &
       \sigma_4\\
       14 & 34 & - 6 & 86 & - 18 & 38 & \sigma_2 & \sigma_3 & \sigma_4 &
       \sigma_5\\
       15 & - 6 & 47 & - 18 & 38 & 51 & \sigma_3 & \sigma_4 & \sigma_5 &
       \sigma_6\\
       - 52 & - 160 & 86 & \sigma_1 & \sigma_2 & \sigma_3 & \sigma_7 &
       \sigma_8 & \sigma_9 & \sigma_{10}\\
       34 & 86 & - 18 & \sigma_2 & \sigma_3 & \sigma_4 & \sigma_8 & \sigma_9 &
       \sigma_{10} & \sigma_{11}\\
       - 6 & - 18 & 38 & \sigma_3 & \sigma_4 & \sigma_5 & \sigma_9 &
       \sigma_{10} & \sigma_{11} & \sigma_{12}\\
       47 & 38 & 51 & \sigma_4 & \sigma_5 & \sigma_6 & \sigma_{10} &
       \sigma_{11} & \sigma_{12} & \sigma_{13}
     \end{array} \right] \]
 }%
 where $\sigma_1 = \sigma_{5, 0} \nocomma, \sigma_2 = \sigma_{4, 1}, \sigma_2
  = \sigma_{3, 2}, \sigma_4 = \sigma_{2, 3}, \sigma_5 = \sigma_{1, 4},
  \sigma_6 = \sigma_{0, 5}$, $\sigma_7 = \sigma_{6, 0} \nocomma, \sigma_8 =
  \sigma_{5, 1}, \sigma_9 = \sigma_{4, 2}, \sigma_{10} = \sigma_{3, 3},
  \sigma_{11} = \sigma_{2, 4}, \sigma_{12} = \sigma_{1, 5}, \sigma_{13} =
  \sigma_{0, 6}$.

  The first $6 \times 6$ diagonal block $H^{B, B}_{\sigma}$ is invertible. To
  have a flat extension $H_{\sigma}^{B^+, B^+}$, we impose  the condition that the sixteen $7
  \times 7$ minors of $H_{\sigma}^{B^+, B^+}$, which contains the first $6$
  rows and columns, must vanish. This yields the following system of quadratic
  equations:

  {\scriptsize \[ \left\{ \begin{array}{l}
       - 814592 \hspace{0.17em} \sigma_1^2 - 1351680 \hspace{0.17em} \sigma_1
       \sigma_2 - 476864 \hspace{0.17em} \sigma_1 \sigma_3 - 599040
       \hspace{0.17em} \sigma_2^2 - 301440 \hspace{0.17em} \sigma_2 \sigma_3 -
       35072 \hspace{0.17em} \sigma_3^2\\
       \hspace{0.5cm} - 520892032 \hspace{0.17em} \sigma_1 - 396821760
       \hspace{0.17em} \sigma_2 - 164529152 \hspace{0.17em} \sigma_3 + 1693440
       \hspace{0.17em} \sigma_7 - 86394672128 = 0\\
       - 814592 \hspace{0.17em} \sigma_2^2 - 1351680 \hspace{0.17em} \sigma_2
       \sigma_3 - 476864 \hspace{0.17em} \sigma_2 \sigma_4 - 599040
       \hspace{0.17em} \sigma_3^2 - 301440 \hspace{0.17em} \sigma_3 \sigma_4 -
       35072 \hspace{0.17em} \sigma_4^2\\
       \hspace{0.5cm} + 335275392 \hspace{0.17em} \sigma_2 + 257276160
       \hspace{0.17em} \sigma_3 + 96277632 \hspace{0.17em} \sigma_4 + 1693440
       \hspace{0.17em} \sigma_9 - 34904464128 = 0\\
       - 814592 \hspace{0.17em} \sigma_3^2 - 1351680 \hspace{0.17em} \sigma_3
       \sigma_4 - 476864 \hspace{0.17em} \sigma_3 \sigma_5 - 599040
       \hspace{0.17em} \sigma_4^2 - 301440 \hspace{0.17em} \sigma_4 \sigma_5 -
       35072 \hspace{0.17em} \sigma_5^2\\
       \hspace{0.5cm} + 13226880 \hspace{0.17em} \sigma_3 + 13282560
       \hspace{0.17em} \sigma_4 - 8227200 \hspace{0.17em} \sigma_5 + 1693440
       \hspace{0.17em} \sigma_{11} + 31714560 = 0\\
       - 814592 \hspace{0.17em} \sigma_4^2 - 1351680 \hspace{0.17em} \sigma_4
       \sigma_5 - 476864 \hspace{0.17em} \sigma_4 \sigma_6 - 599040
       \hspace{0.17em} \sigma_5^2 - 301440 \hspace{0.17em} \sigma_5 \sigma_6 -
       35072 \hspace{0.17em} \sigma_6^2\\
       \hspace{0.5cm} + 212860736 \hspace{0.17em} \sigma_4 + 162698880
       \hspace{0.17em} \sigma_5 + 51427456 \hspace{0.17em} \sigma_6 + 1693440
       \hspace{0.17em} \sigma_{13} - 13356881792 = 0\\
       - 814592 \hspace{0.17em} \sigma_1 \sigma_2 - 675840 \hspace{0.17em}
       \sigma_1 \sigma_3 - 238432 \hspace{0.17em} \sigma_1 \sigma_4 - 675840
       \hspace{0.17em} \sigma_2^2 \\
                            \hspace{0.5cm}- 837472 \hspace{0.17em} \sigma_2
                            \sigma_3  - 150720 \hspace{0.17em} \sigma_2 \sigma_4 - 150720 \hspace{0.17em}
       \sigma_3^2 - 35072 \hspace{0.17em} \sigma_3 \sigma_4
       + 167637696 \hspace{0.17em} \sigma_1 - 131807936
       \hspace{0.17em} \sigma_2 \\ \hspace{0.5cm} - 150272064 \hspace{0.17em} \sigma_3 -
       82264576 \hspace{0.17em} \sigma_4 + 1693440 \hspace{0.17em} \sigma_8 +
       54990043008 = 0\\
       - 814592 \hspace{0.17em} \sigma_2 \sigma_3 - 675840 \hspace{0.17em}
       \sigma_2 \sigma_4 - 238432 \hspace{0.17em} \sigma_2 \sigma_5 - 675840
       \hspace{0.17em} \sigma_3^2 - 837472 \hspace{0.17em} \sigma_3 \sigma_4 -
       150720 \hspace{0.17em} \sigma_3 \sigma_5 - 150720 \hspace{0.17em}
       \sigma_4^2 \\
       \hspace{0.5cm}- 35072 \hspace{0.17em} \sigma_4 \sigma_5 + 6613440 \hspace{0.17em} \sigma_2 + 174278976
       \hspace{0.17em} \sigma_3 + 124524480 \hspace{0.17em} \sigma_4 +
       48138816 \hspace{0.17em} \sigma_5 + 1693440 \hspace{0.17em} \sigma_{10}
       - 746438400 = 0\\
       - 814592 \hspace{0.17em} \sigma_3 \sigma_4 - 675840 \hspace{0.17em}
       \sigma_3 \sigma_5 - 238432 \hspace{0.17em} \sigma_3 \sigma_6 - 675840
       \hspace{0.17em} \sigma_4^2 - 837472 \hspace{0.17em} \sigma_4 \sigma_5 -
       150720 \hspace{0.17em} \sigma_4 \sigma_6 - 150720 \hspace{0.17em}
       \sigma_5^2\\
       \hspace{0.5cm} - 35072 \hspace{0.17em} \sigma_5 \sigma_6 + 106430368
       \hspace{0.17em} \sigma_3 + 87962880 \hspace{0.17em} \sigma_4 + 32355008
       \hspace{0.17em} \sigma_5 - 4113600 \hspace{0.17em} \sigma_6 + 1693440
       \hspace{0.17em} \sigma_{12} - 183201600 = 0\\
       - 814592 \hspace{0.17em} \sigma_2 \sigma_4 - 675840 \hspace{0.17em}
       \sigma_2 \sigma_5 - 238432 \hspace{0.17em} \sigma_2 \sigma_6 - 675840
       \hspace{0.17em} \sigma_3 \sigma_4 - 599040 \hspace{0.17em} \sigma_3
       \sigma_5 - 150720 \hspace{0.17em} \sigma_3 \sigma_6 - 238432
       \hspace{0.17em} \sigma_4^2 \\
       \hspace{0.5cm} - 150720 \hspace{0.17em} \sigma_4 \sigma_5 -
       35072 \hspace{0.17em} \sigma_4 \sigma_6+ 106430368 \hspace{0.17em} \sigma_2 + 81349440
       \hspace{0.17em} \sigma_3 + 193351424 \hspace{0.17em} \sigma_4 +
       128638080 \hspace{0.17em} \sigma_5 + 48138816 \hspace{0.17em} \sigma_6
       \\
       \hspace{0.5cm} + 1693440 \hspace{0.17em} \sigma_{11} - 21721986624 = 0\\
       - 814592 \hspace{0.17em} \sigma_1 \sigma_3 - 675840 \hspace{0.17em}
       \sigma_1 \sigma_4 - 238432 \hspace{0.17em} \sigma_1 \sigma_5 - 675840
       \hspace{0.17em} \sigma_2 \sigma_3 - 599040 \hspace{0.17em} \sigma_2
       \sigma_4 - 150720 \hspace{0.17em} \sigma_2 \sigma_5 - 238432
       \hspace{0.17em} \sigma_3^2 \\
       \hspace{0.5cm} - 150720 \hspace{0.17em} \sigma_3 \sigma_4 -
       35072 \hspace{0.17em} \sigma_3 \sigma_5+ 6613440 \hspace{0.17em} \sigma_1 + 6641280
       \hspace{0.17em} \sigma_2 - 264559616 \hspace{0.17em} \sigma_3 -
       198410880 \hspace{0.17em} \sigma_4 - 82264576 \hspace{0.17em} \sigma_5\\
       \hspace{0.5cm}
       + 1693440 \hspace{0.17em} \sigma_9 + 1312368000 = 0\\
       - 814592 \hspace{0.17em} \sigma_1 \sigma_4 - 675840 \hspace{0.17em}
       \sigma_1 \sigma_5 - 238432 \hspace{0.17em} \sigma_1 \sigma_6 - 675840
       \hspace{0.17em} \sigma_2 \sigma_4 - 599040 \hspace{0.17em} \sigma_2
       \sigma_5 - 150720 \hspace{0.17em} \sigma_2 \sigma_6 - 238432
       \hspace{0.17em} \sigma_3 \sigma_4 \\
       \hspace{0.5cm} - 150720 \hspace{0.17em} \sigma_3
       \sigma_5 - 35072 \hspace{0.17em} \sigma_3 \sigma_6 + 106430368 \hspace{0.17em} \sigma_1 + 81349440
       \hspace{0.17em} \sigma_2 + 25713728 \hspace{0.17em} \sigma_3 -
       260446016 \hspace{0.17em} \sigma_4 \\
       \hspace{0.5cm} - 198410880 \hspace{0.17em} \sigma_5
       - 82264576 \hspace{0.17em} \sigma_6 + 1693440 \hspace{0.17em}
       \sigma_{10} + 34550702464 = 0
     \end{array} \right. \]}

  The set of solutions of this system is an algebraic variety of dimension $3$
  and degree $52$. A solution is

  {\scriptsize $\sigma_1 = - 484$, $\sigma_2 = 226$, $\sigma_3 = - 54$,
  $\sigma_4 = 82$, $\sigma_5 = - 6$, $\sigma_6 = 167$, $\sigma_7 = - 1456$,
  $\sigma_8 = 614$, $\sigma_9 = - 162$, $\sigma_{10} = 182$, $\sigma_{11} = -
  18$, $\sigma_{12} = 134$, $\sigma_{13} = 195$.}
\end{example}

\subsection{Computing an orthogonal basis of $\mathcal{A}_{\sigma}$}

In this section, we describe a new method to construct a basis $B$ of
$\mathcal{A}_{\sigma}$ and to detect flat extensions, from the knowledge of
the moments $\sigma_{\alpha}$ of $\sigma (\mathbf{z})$. We are going to
inductively construct a family $P$ of polynomials, orthogonal for the inner
product
\[ (p, q) \mapsto \langle p, q \rangle_{\sigma} \assign \langle \sigma \mid p
   q \rangle, \]
and a monomial set $B$ connected to $1$ such that $\langle B \rangle = \langle
P \rangle$.

We start with $B = \{ 1 \}$, $P = \{ 1 \} \subset R$. As $\langle 1, 1
\rangle_{\sigma} = \langle \sigma \mid 1 \rangle \neq 0$, the family $P$ is
orthogonal for $\sigma$ and $\langle B \rangle = \langle P \rangle$.

We now describe the induction step. \ Assume that we have a set $B = \{ m_1,
\ldots, m_s \}$ and $P = \{ p_1, \ldots, p_s \}$ such that
\begin{itemize}
  \item $\langle B \rangle = \langle P \rangle$

  \item $\langle p_i, p_j \rangle_{\sigma} \neq 0$ if $i = j$ and $0$
  otherwise.
\end{itemize}
To construct the next orthogonal polynomials, we consider the monomials in
$\partial B = \{ m'_1, \ldots, m'_l \}$ and project them on $\langle P
\rangle$:
\[ p'_i = m_i' - \sum_{j = 1}^s \frac{\langle m_i', p_j
   \rangle_{\sigma}}{\langle p_j, p_j \rangle_{\sigma}} p_j, i = 1, \ldots, l.
\]
By construction, $\langle p_i', p_j \rangle_{\sigma} = 0$ and $\langle p_1,
\ldots, p_s, p_i' \rangle = \langle m_1, \ldots, m_s, m'_i \rangle$. \ We
extend $B$ by choosing a subset of monomials $B' = \{ m'_{i_1}, \ldots,
m'_{i_k} \}$ such that the matrix
\[ [\langle p'_{i_j}, p'_{i_{j'}} \rangle_{\sigma}]_{1 \leqslant j, j'
   \leqslant k} \]
is invertible. The family $P$ is then extended by adding an orthogonal family
of polynomials $\{ p_{s + 1}, \ldots, p_{s + k} \}$ constructed from $\{
p'_{i_1}, \ldots, p'_{i_k} \}$. If all the polynomials $p'_i$ are such that
$\langle p'_i, p'_j \rangle_{\sigma} = 0$, the process stops.

This leads to the following algorithm:

{\algorithm{\label{algo:1}{\textbf{Input:}} the coefficients
$\sigma_{\alpha}$ of a series $\sigma \in \mathbb{C} [[\mathbf{z}]]$ for
$\alpha \in A \subset \mathbb{N}^n$ connected to $1$ with $\sigma_0 \neq 0$.
\begin{itemize}
  \item Let $B \assign \{ 1 \}$; \ $P = \{ 1 \} ;$ r := 1; $E = \langle
  \mathbf{z}^{\alpha} \rangle_{\alpha \in A}$;

  \item While $s > 0$ and $B^+ \cdummy B^+ \subset E$ do \ {\textbf{}}
  \begin{itemize}
    \item Compute $\partial B = \{ m'_1, \ldots, m'_l \}$ and $p'_i = m_i' -
    \sum_{j = 1}^s \frac{\langle m_i', p_j \rangle_{\sigma}}{\langle p_j, p_j
    \rangle_{\sigma}} p_j$;

    \item Compute a (maximal) \ subset $B' = \{ m'_{i_1}, \ldots, m'_{i_k} \}$
    of $\partial B$ such that $[\langle p'_{i_j}, p'_{i_{j'}}
    \rangle_{\sigma}]_{1 \leqslant j, j' \leqslant k}$ is invertible.

    \item Compute an orthogonal family of polynomials $\{ p_{s + 1}, \ldots,
    p_{s + k} \}$ \ from $\{ p'_{i_1}, \ldots, p'_{i_k} \}$.

    \item $B : = B \cup B'$, $P : = P \cup \{ p_{s + 1}, \ldots, p_{s + k}
    \}$; $r \plusassign k$;
  \end{itemize}
  \item If $B^+ \cdummy B^+ \not\subset E$ then return \tmverbatim{failed}.
\end{itemize}
{\textbf{Output:}} \tmverbatim{failed} or \tmverbatim{success} with \
\begin{itemize}
  \item a set of monomials $B = \{ m_1, \ldots, m_r \}$ connected to $1$, and
  non-degenerate for \ $\langle \cdummy, \cdummy \rangle_{\sigma}$.

  \item a set of polynomials $P = \{ p_1, \ldots, p_r \}$ orthogonal for
  $\sigma$ and such that $\langle B \rangle = \langle P \rangle$.

  \item the relations $\rho_i \assign m_i' - \sum_{j = 1}^s \frac{\langle
  m_i', p_j \rangle_{\sigma}}{\langle p_j, p_j \rangle_{\sigma}} p_j$ for the
  monomials $m'_i$ in \ $\partial B = \{ m'_1, \ldots, m'_l \}$.
\end{itemize}}}

The above algorithm is a Gramm-Schmidt-type orthogonalization method, where, at
each step, new monomials are taken in $\partial B$ and projected onto the
space spanned by the previous monomial set $B$.
Notice that if the polynomials $p_i$ are of degree at most $d' < d$, then
only the moments of $\sigma$ of degree $\leqslant 2 d' + 1$ {\textbf{}}are
involved in this computation.

\begin{proposition}
  \label{thm:flatextalgo}If Algorithm \ref{algo:1} outputs with success a set
  $B = \{ m_1, \ldots, m_r \}$ and the relations $\rho_i \assign m_i' -
  \sum_{j = 1}^s \frac{\langle m_i', p_j \rangle_{\sigma}}{\langle p_j, p_j
  \rangle_{\sigma}} p_j$, for $m'_i$ in $\partial B = \{ m'_1, \ldots, m'_l
  \}$, then $\sigma$ coincides on $\langle B^{\noplus +} \cdot B^+ \rangle$
  with the series $\tilde{\sigma}$ such that
  \begin{itemize}
    \item $\rank H_{\tilde{\sigma}} = r$;

    \item $B$ and P are basis of $\mathcal{A}_{\tilde{\sigma}}$ for the inner
    product $\langle \cdummy, \cdummy \rangle_{\tilde{\sigma}}$; {\textbf{}}

    \item The ideal $I_{\tilde{\sigma}} = \ker H_{\bar{\sigma}}$ is generated
    by $(\rho_i)_{i = 1, \ldots, l} $;

    \item The matrix of multiplication by $x_k$ in the basis $P$ of
    $\mathcal{A}_{\tilde{\sigma}}$ is
    \[ M_k = \left( \frac{\langle \sigma \mid x_k  \nobracket p_i p_j \rangle
       \nobracket}{\langle \sigma \mid \nobracket p^2_j \rangle \nobracket}
       \right)_{1 \leqslant i, j \leqslant r} . \]
  \end{itemize}
\end{proposition}

\begin{proof}
  By construction, $B$ is connected to $1$. A basis $B'$ of $\langle B^+
  \rangle$ is formed by the elements of $B$ and the polynomials $\rho_i
  \nocomma, i = 1, \ldots, l$. Since Algorithm \ref{algo:1} \ stops with
  success, we have $\forall i, j \in [1, l]$, $\nocomma \forall b \in \langle
  B \rangle$, $\langle \nobracket \rho_i, b \rangle_{\sigma} \nobracket =
  \langle \nobracket \rho_i, \rho_j \rangle_{\sigma} \nobracket = 0$ and
  $\rho_1, \ldots, \rho_l \in \ker H_{\sigma}^{B^+, B^+}$. As $\langle B^+
  \rangle = \langle B \rangle \oplus \langle \rho_1, \ldots, \rho_l \rangle$,
  $\rank H_{\sigma}^{B^+, B^+} = \rank H^{B, B}_{\sigma}$ and
  $H_{\sigma}^{B^+, B^+}$ is a flat extension of $H_{\sigma}^{B, B}$. By
  construction, $P$ is an orthogonal basis of $\langle B \rangle$ and the
  matrix of $H_{\sigma}^{B, B}$ in this basis is diagonal with non-zero
  entries on the diagonal. Thus $H_{\sigma}^{B, B}$ is of rank $r$.

  By Theorem \ref{thm:flat:extension}, $\sigma$ coincides on $\langle
  B^{\noplus +} \cdot B^+ \rangle$ with a series $\tilde{\sigma} \in R^{\ast}$
  such that $B$ is a basis of $\mathcal{A}_{\bar{\sigma}} = R /
  I_{\tilde{\sigma}}$ and $I_{\tilde{\sigma}} = (\ker H_{\tilde{\sigma}}^{B^+,
  B^+}) = (\rho_1, \ldots, \rho_l)_{}$.

  As $\langle B^{\noplus +} \rangle = \langle B^{} \rangle \oplus \langle
  \rho_1, \ldots, \rho_l \rangle = \langle P^{} \rangle \oplus \langle \rho_1,
  \ldots, \rho_l \rangle$ and $P$ is an orthogonal basis of
  $\mathcal{A}_{\bar{\sigma}}$, which is orthogonal to $\langle \rho_1,
  \ldots, \rho_l \rangle$, we have
  \[ x_k p_i = \sum_{j = 1}^r \frac{\langle \sigma \mid x_k  \nobracket p_i
     p_j \rangle \nobracket}{\langle \sigma \mid \nobracket p^2_j \rangle
     \nobracket} p_j + \rho \]
  with $\rho \in \langle \rho_1, \ldots, \rho_l \rangle$. This shows that the
  matrix of the multiplication by $x_k$ modulo $I_{\bar{\sigma}} = (\rho_1,
  \ldots, \rho_l)$, in the basis $P = \{ p_1, \ldots, p_r \}$ is \ $M_k =
  \left( \frac{\langle \sigma \mid x_k  \nobracket p_i p_j \rangle
  \nobracket}{\langle \sigma \mid \nobracket p^2_j \rangle \nobracket}
  \right)_{1 \leqslant i, j \leqslant r}$.
\end{proof}

\begin{remark}
  It can be shown that the polynomials $(\rho_i)_{i = 1, \ldots, l} $ are a
  border basis of $I_{\tilde{\sigma}}$ for the basis $B$
  {\cite{ML08,BracCMT09:laa,Mou99,Mourrain2005}}.
\end{remark}

\begin{remark}
  If $H^{B, B}_{\sigma} \succcurlyeq 0$, then by Proposition \ref{prop:2}, the
  common roots $\zeta_1, \ldots, \zeta_r$ of the polynomials $\rho_1, \ldots,
  \rho_l$ are simple and real $\in \mathbb{R}^n$. They are the cubature
  points:
  \[ \bar{\sigma} = \sum_{i = 1}^r w_j \hspace{0.17em} \mathbf{e}_{\zeta_j}
     \nocomma, \]
  with $w_j > 0$.
\end{remark}

\section{The cubature formula from the moment matrix}\label{sec:root}

We now describe how to recover the cubature formula, from the
moment matrix $H_{\sigma}^{B^+, B'^+}$. We assume that the flat extension
condition is satisfied:
\begin{equation}
  \rank \hspace{0.17em} H_{\sigma}^{B^+, B'^+} = H_{\sigma}^{B, B'^{}} = |B| =
  |B' | .
\end{equation}
\begin{theorem}
  \label{thm:14}Let $B$ and $B'$ be monomial subsets of $R$ of size $r$
  connected to $1$ and $\sigma \in \langle B^+ \cdot B^{' +} \rangle^{\ast}$.
  Suppose that $\rank H_{\sigma}^{B^+, B'^+} = \rank \hspace{0.17em}
  H_{\sigma}^{B, B'} = |B| = |B' |$. Let $M_i = [H_{\sigma}^{B, B'}]^{- 1}
  [H_{\sigma}^{x_i B, B'}]$ and $(M'_i)^t = [H_{\sigma}^{B, x_i B'}]
  [H_{\sigma}^{B, B'}]^{- 1}$. Then,
  \begin{enumerate}
    \item $B$ and $B'$ are bases of $\Ac_{\tilde{\sigma}} = R / \ker
    H_{\tilde{\sigma}}$,

    \item $M_i$ (resp. $M'_i$) is the matrix of multiplication by $x_i$ in the
    basis $B$ (resp. $B'$) of $\Ac_{\tilde{\sigma}}$,
  \end{enumerate}
\end{theorem}

\begin{proof}
  By the flat extension theorem \ref{thm:flat:extension}, there exists
  $\bar{\sigma} \in R^{\ast} \nocomma$ such that $H_{\tilde{\sigma}}$ is a
  flat extension of $H_{\sigma}^{B^+, B'^+}$ of rank $r = | B | = | B' |$ and
  $\ker H_{\bar{\sigma}} = (\ker H_{\sigma}^{B^+, B'^+})$. As $R = \langle B
  \rangle \oplus \ker H_{\bar{\sigma}}$ and $\rank H_{\tilde{\sigma}} =
  r$, \ $\Ac_{\tilde{\sigma}} = R / \ker H_{\bar{\sigma}}$ is of dimension $r$
  and generated by $B$. Thus $B$ is a basis of $\Ac_{\tilde{\sigma}}$. A
  similar argument shows that $B'$ is also a basis of $\Ac_{\tilde{\sigma}}$.
  We denote by $\pi : \Ac_{\tilde{\sigma}} \rightarrow \langle B \rangle$ and
  $\pi' : \Ac_{\tilde{\sigma}} \rightarrow \langle B' \rangle$ the isomorphism
  associated to these basis representations.

  The matrix $[H_{\sigma}^{B, B'}]$ is the matrix of the Hankel operator
  \begin{eqnarray*}
    \bar{H}_{\bar{\sigma}} : \Ac_{\tilde{\sigma}} & \rightarrow &
    \Ac_{\tilde{\sigma}}^{\ast}\\
    a & \mapsto & a \star \bar{\sigma}
  \end{eqnarray*}
  in the basis $B$ and the dual basis of $B'$. Similarly, \ $[H_{\sigma}^{x_i
  B, B'}]$ is the matrix of
  \begin{eqnarray*}
    \bar{H}_{x_{i \star} \bar{\sigma}} : \Ac_{\tilde{\sigma}} & \rightarrow &
    \Ac_{\tilde{\sigma}}^{\ast}\\
    a & \mapsto & a \star x_i \star \bar{\sigma}
  \end{eqnarray*}
  in the same bases. As $x_i \star \bar{\sigma} = \bar{\sigma} \circ M_i$
  where $M_i : \Ac_{\tilde{\sigma}} \rightarrow \Ac_{\tilde{\sigma}} $ is the
  multiplication by $x_i$ in $\Ac_{\tilde{\sigma}}$, we deduce that
  $\bar{H}_{x_{i \star} \bar{\sigma}} = \bar{H}_{\bar{\sigma}} \circ M_i$ and
  $[H_{\sigma}^{B, B'}]^{- 1} [H_{\sigma}^{x_i B, B'}]$ is the matrix of
  multiplication by $x_i$ in the basis $B$ of $\Ac_{\tilde{\sigma}}$. By
  exchanging the role of $B$ and $B'$ and by transposition ($[H_{\sigma}^{B,
  B'}]^t = [H_{\sigma}^{B', B}]$), we obtain that $[H_{\sigma}^{B, x_i B'}]
  [H_{\sigma}^{B, B'}]^{- 1}$ is the transpose of the matrix of multiplication
  by $x_i$ in the basis $B'$ of $\Ac_{\tilde{\sigma}}$.
\end{proof}

\begin{theorem}
  Let $B$ \ be a monomial subset of $R$ of size $r$ connected to $1$ and
  $\sigma \in \langle B^+ \cdot B^+ \rangle^{\ast}$. Suppose that $\rank
  H_{\sigma}^{B^+, B^+} = \rank \hspace{0.17em} H_{\sigma}^{B, B} = |B|$=r and
  that $H_{\sigma}^{B, B} \succcurlyeq 0$. Let $M_i = [H_{\sigma}^{B, B'}]^{-
  1} [H_{\sigma}^{x_i B, B'}]$. Then $\sigma$ can be decomposed as
  \[ \sigma = \sum_{j = 1}^r w_j \hspace{0.17em} \mathbf{e}_{\zeta_j} \]
  with $w_j > 0$ and $\zeta_j \in \mathbb{R}^n$ such that $M_i$ have $r$
  common linearly independent eigenvectors $\mathbf{u}_j$, $j = 1, \ldots,
  r$ and
  \begin{itemize}
    \item $\zeta_{j, i} = \frac{\langle \sigma | \nobracket x_i \mathbf{u}_j
    \rangle}{\langle \sigma | \nobracket \mathbf{u}_j \rangle}$ for $1 \le i
    \le n$, $1 \le j \le r$.

    \

    \item $w_j = \frac{\langle \sigma | \nobracket \mathbf{u}_j \rangle
    }{\nobracket \nobracket \mathbf{u}_j (\zeta_{j, 1}, \ldots, \zeta_{j,
    n})}$.
  \end{itemize}
\end{theorem}

\begin{proof}
  By theorem \ref{thm:14}, the matrix $M_i$ is the matrix of multiplication
  by $x_i$ in the basis $B$ of $\Ac_{\tilde{\sigma}}$. As $H_{\sigma}^{B, B}
  \succcurlyeq 0$, the flat extension theorem \ref{thm:flat:extension} implies
  that $H_{\bar{\sigma}} \succcurlyeq 0$ and that
  \[ \bar{\sigma} = \sum_{j = 1}^r w_j \hspace{0.17em} \mathbf{e}_{\zeta_j}
  \]
  where $w_j > 0$ and $\zeta_j \in \mathbb{R}^n$ are the simple roots of the
  ideal $\ker H_{\bar{\sigma}}$. Thus the commuting operators $M_i$ are
  diagonalizable in a common basis of eigenvectors $\mathbf{u}_i$, $i = 1,
  \ldots, r$, which are scalar multiples of the interpolation polynomials at
  the roots $\zeta_1, \ldots, \zeta_r$: $\mathbf{u}_i (\zeta_i) = \lambda_i
  \neq 0$ and \ $\mathbf{u}_i (\zeta_j) = 0$ if $j \neq i$ (see
  {\cite{em-07-irsea}}[Chap. 4] or {\cite{DiEm05:Cox}}). We deduce that
  \[ \langle \bar{\sigma} \mid \mathbf{u}_j \rangle = \sum_{k = 1}^r w_k
     \hspace{0.17em} \mathbf{u}_j (\zeta_k) = w_j \lambda_j \ and \
     \langle \bar{\sigma} \mid x_i \mathbf{u}_j \rangle = \zeta_{j, i} w_j
     \lambda_j, \]
  so that $\zeta_{j, i} = \frac{\langle \sigma | \nobracket x_i \mathbf{u}_j
  \rangle}{\langle \sigma | \nobracket \mathbf{u}_j \rangle} \nosymbol$. As
  $\mathbf{u}_i (\zeta_i) = \lambda_i$, we have $w_j = \frac{\langle \sigma
  | \nobracket \mathbf{u}_j \rangle }{\nobracket \nobracket \mathbf{u}_j
  (\zeta_{j, 1}, \ldots, \zeta_{j, n})}$.
\end{proof}

{\algorithm{\label{algo:dec}\textbf{Input:} $B$ is a set of monomials
    connected to $1$, $\sigma\in \langle
    B^{+}\cdot B^{+}\rangle^{*}$ such that  $H_{\sigma}^{B^+, B^+}$ is a
    flat extension of $H_{\sigma}^{B, B}$ of rank $|B|$.
\begin{itemize}
  \item Compute an orthogonal basis $\{ p_1, \ldots, p_r \}$ of $B$ for
  $\sigma$;

  \item Compute the matrices $M_k = \left( \frac{\langle \sigma \mid x_k
  \nobracket p_i p_j \rangle \nobracket}{\langle \sigma \mid \nobracket p^2_j
  \rangle \nobracket} \right)_{1 \leqslant i, j \leqslant r}$;

  \item Compute their common eigenvectors $\mathbf{u}_{_1}, \ldots,
  \mathbf{u}_r ;$
\end{itemize}
\textbf{Output}: For $j = 1, \ldots, r$,
\begin{itemize}
  \item $\zeta_j = \left( \frac{\langle \sigma | \nobracket x_1 \mathbf{u}_j
  \rangle}{\langle \sigma | \nobracket \mathbf{u}_j \rangle}, \ldots,
  \frac{\langle \sigma | \nobracket x_{{in}} \mathbf{u}_j
  \rangle}{\langle \sigma | \nobracket \mathbf{u}_j \rangle} \right)$

  \item $w_j = \frac{\langle \sigma | \nobracket \mathbf{u}_j \rangle
  }{\nobracket \nobracket \mathbf{u}_j (\zeta_{j, 1}, \ldots, \zeta_{j,
  n})}$
\end{itemize}}}

\begin{remark}
  Since the matrices $M_k$ commute and are diagonalizable with the same basis,
  their common eigenvectors can be obtained by computing the eigenvectors of a
  generic linear combination $l_1 M_1 + \cdots + l_n M_n$, $l_i \in
  \mathbb{R}$.
\end{remark}

\section{Cubature formula by convex optimization}

As described in the previous section, the computation of cubature formulae
reduces to a{\textbf{}} low rank Hankel matrix completion problem, using
the{\textbf{}} flat extension property. In this section, we describe a new
approach which relaxes this problem into a convex optimization problem.

Let $V \subset R$ be a vector space spanned by monomials
$\mathbf{x}^{\alpha}$ for $\alpha \in A \subset \mathbb{N}^n$. Our aim is
to construct a cubature formula for an integral function $I$ exact on $V$. Let
$\mathbf{i}= (I [\mathbf{x}^{\alpha}])_{\alpha \in A}$ be the sequence of
moments given by the integral $I$. We also denote $\mathbf{i} \in V^{\ast}$
the associated linear form such that $\forall v \in V \langle \mathbf{i}
\mid v \rangle = I [v]$.

For $k \in \mathbb{N}$, we denote by
\[ \mathcal{H}^k (\mathbf{i}) = \{ H_{\sigma} \mid \sigma \in R_{2
   k}^{\ast}, \sigma_{\alpha} =\mathbf{i}_{\alpha}\  \mathrm{for}\ \alpha \in A,
   H_{\sigma} \succcurlyeq 0 \}, \]
the set of semi-definite Hankel operators on $R_t$ is associated to moment
sequences which extend $\mathbf{i}$. We can easily check that $\mathcal{H}^k
(\mathbf{i})$ is a convex set. We denote by $\mathcal{H}_r^k (\mathbf{i})$
the set of elements of $\mathcal{H}^k (\mathbf{i})$ of rank $\leqslant r$.

A subset of $\mathcal{H}_r^k (\mathbf{i})$ is the set of Hankel operators
associated to cubature formulae of $r$ points:
\[ \mathcal{E}^k_r (\mathbf{i}) = \left\{ H_{\sigma} \in \mathcal{H}^k
   (\mathbf{i}) \mid \sigma = \sum_{i = 1}^r w_i \hspace{0.17em}
   \mathbf{e}_{\zeta_i}, \omega_i > 0, \zeta_i \in \mathbb{R}^n \right\} .
\]
We can check that $\mathcal{E}^k (\mathbf{i}) = \cup_{r \in \mathbb{N}}
\mathcal{E}^k_r (\mathbf{i})$ is also a convex set.

To impose the cubature points to be in a semialgebraic set $\mathcal{S}$
defined by equality \ and inequalities \ $\mathcal{S}= \{ \mathbf{x} \in
\mathbb{R}^n \mid g_1^0 (\mathbf{x}) = 0, \ldots, g_{n_1}^0 (\mathbf{x})
= 0, g_1^+ (\mathbf{x}) \geqslant 0, \ldots, g_{n_2}^+ (\mathbf{x})
\geqslant 0 \}$, one can refine the space of $\mathcal{H}^k (\mathbf{i})$ by
imposing that $\sigma$ is positive on the quadratic module (resp. preordering)
associated to the constraints $\mathbf{g}= \{ g_1^0, \ldots, g_n^0 ; g_1^+,
\ldots, g_n^+ \}$ {\cite{Lasserre:book}}. For the sake of simplicity, we don't
analyze this case here, which can be done in a similar way.

The Hankel operator $H_{\sigma} \in \mathcal{E}^k_r (\mathbf{i})$ associated
to a cubature formula of $r$ points is an element of $\mathcal{H}_r^k
(\mathbf{i})$. In order to find a cubature formula of minimal rank, we would
like to compute a minimizer solution of the following optimization problem:
\[ \min_{H \in \mathcal{H}^k (\mathbf{i})} \rank (H) \]
However this problem is NP-hard {\cite{Fazel2002}}. We therefore relax it into the
minimization of the nuclear norm of the Hankel operators, i.e. the
minimization of the sum of the singular values of the Hankel matrix {\cite{RFP2010}}. More
precisely, for a generic matrix $P \in \mathbb{R}^{s_t \times s_t}$, we
consider the following minimization problem:
\begin{equation}
  \min_{H \in \mathcal{H}^k (\mathbf{i})} \trace (P^t H P)
  \label{eq:min:norm}
\end{equation}
Let $(A, B) \in \mathbb{R}^{s_k \times s_k} \times \mathbb{R}^{s_k \times
s_k} \rightarrow \langle A, B \rangle = \trace (A B)$ denote the inner
product induced by the trace on the space of $s_k \times s_k$ matrices. The
optimization problem (\ref{eq:min:norm}) requires  minimizing the linear
form $H \rightarrow \trace (H P P^t) = \langle H, P P^t \rangle$ on the
convex set $\mathcal{H}^k (\mathbf{i})$. As the trace of $P^t H P$ is
bounded by below by $0$ when $H \succcurlyeq 0$ , our optimization problem
(\ref{eq:min:norm}) has a non-negative minimum $\geqslant 0$.

Problem (\ref{eq:min:norm}) is a Semi-Definite Program (SDP), which can be
solved efficiently by interior point methods. See
{\cite{NesterovNemirovski94}}. SDP is an important ingredient of relaxation
techniques in polynomial optimization. See {\cite{Lasserre:book,LaurentSurvey09}}.

Let $\Sigma^k = \left\{ p = \sum_{i = 1}^l p_i^2 \mid p_i \in R_k \right\}$ be
the set of polynomials of degree $\leqslant 2 k$ which are sums of squares,
let $\mathbf{x}^{(k)}$ be the vector of all monomials in $\mathbf{x}$ of
degree $\leqslant k$ and let $q (\mathbf{x}) = (\mathbf{x}^{(k)})^t P P^t
\hspace{1em} \mathbf{x}^{(k)} \in \Sigma^k$. Let $p_i (\mathbf{x})$ ($1
\leqslant i \leqslant s_k$) denote the polynomial $\langle P_i,
\mathbf{x}^{(k)} \rangle$ associated to the column $P_i$ of $P$. We have $q
(\mathbf{x}) = \sum_{i = 1}^{s_k} p_i (\mathbf{x})^2$ and for any $\sigma
\in R_{2 k}$,
\[ \trace (P^t H_{\sigma} P) = \langle H_{\sigma} \mid P P^t \rangle =
   \langle \sigma \mid q (\mathbf{x}) \rangle . \]
For any $l \in \mathbb{N}$, we denote by $\pi_l : R_l \rightarrow R_l$ the
linear map which associates to a polynomial $p \in R_l$ its homogeneous
component of degree $l$. We say that $P$ is a proper matrix if $\pi_{2 k} (q
(\mathbf{x})) \neq 0$ for all $\mathbf{x} \in \mathbb{R}^n$.

We are thus looking for cubature formula with a small number of points, which
correspond to Hankel operators with small rank. The next results describes the
structure of truncated Hankel operators, when the degree of truncation is high
enough, compared to the rank.

\begin{theorem}
  \label{thm:flat:trunc}Let $\sigma \in R_{2 k}^{\ast}$ and let $H_{\sigma}$
  be its truncated Hankel operator on $R_k$. If $H_{\sigma}$ is of rank $r
  \leqslant k$, then
  \[ \sigma \equiv \sum_{i = 1}^{r'} \omega_i \mathbf{e}_{\zeta_i} + \sum_{i
     = r' + 1}^r w_i \mathbf{e}_{\zeta_i} \circ \pi_{2 k}  \]
  with $\omega_i \in \mathbb{C} \setminus \{ 0 \}$, $\zeta_i \in
  \mathbb{C}^n$ distinct. If moreover $H_{\sigma} \succcurlyeq 0$, then
  $\omega_i > 0$ and $\zeta_i \in \mathbb{R}^n$ for $i = 1, \ldots, r$.
\end{theorem}

\begin{proof}
  The substitution $\tau_0 : S_{[2 k]} \rightarrow R_{2 k}$ which replaces
  $x_0$ by $1$ is an isomorphism of $\mathbb{K}$-vector spaces. Let
  $\tau_0^{\ast} : R_{2 k}^{\ast} \rightarrow S_{[2 k]}^{\ast}$ be the
  pull-back map on the dual ($\tau_0^{\ast} (\sigma) = \sigma \circ \tau_0$).
  Let $\overline{\sigma} = \sigma \circ \tau_0 = \tau_0^{\ast} (\sigma) \in
  S_{[2 k]}^{\ast}$ be the linear form induced by $\sigma$ on $S_{[2 k]}$ and
  let $H_{\bar{\sigma}} : S_{[k]} \rightarrow S_{[k]}^{\ast}$ be the
  corresponding truncated operator on $S_{[k]}$. The kernel $\bar{K}$ of
  $H_{\bar{\sigma}}$ is the vector space spanned by the homogenization in
  $x_0$ of the elements of the kernel $K$ of $H_{\sigma}$.

  Let $\succ$ be the lexicographic ordering such that $x_0 \succ \cdots \succ
  x_n$. By {\cite{Eis94}}[Theorem 15.20, p. 351], after a generic change of
  coordinates, the initial $J$ of the homogeneous ideal $(\bar{K}) \subset S$
  is Borel fixed. That is, if $x_i \mathbf{x}^{\alpha} \in J$, then $x_j
  \mathbf{x}^{\alpha} \in J$ for $j > i$. Let $\bar{B}$ be the set of
  monomials of degree $k$, which are not in $J$.  Note $J$ is Borel fixed and
  different from $S_{[2 k]}$, $x_0^k \in \bar{B}$. Similarly we check that if
  $x_0^{\alpha_0} \cdots x_n^{\alpha_n} \in J$ with $\alpha_1 = \cdots =
  \alpha_{l - 1} = 0$, then $x_0^{\alpha_0 + 1} x_l^{\alpha_l - 1} \cdots
  x_n^{\alpha_n} \in J$. This shows that $B = \tau_0 (\bar{B})$ is connected
  to $1$.

  As $\langle \bar{B} \rangle \oplus \langle J \rangle = \langle \bar{B}
  \rangle \oplus \bar{K} = S_{[k]}$ where $\bar{K} = \ker H_{\bar{\sigma}}$,
  we have $| B | = r$. As $B$ is connected to $1$, $\deg (B) < r \leqslant k$
  and $B^+ \subset R_k$.

  By the substitution $x_0 = 1$, we have $R_k = \langle B \rangle \oplus K$
  with $K = \ker H_{\sigma}$. Therefore, $H_{\sigma}$ is a flat extension of
  $H_{\sigma}^{B, B}$. Via the flat extension theorem \ref{thm:flat:extension},
  there exists $\lambda_i \in \mathbb{K} \setminus \{ 0 \}$, $\bar{\zeta}_i =
  (\zeta_{i, 0}, \zeta_{i, 1}, \ldots, \zeta_{i, n}) \in \mathbb{K}^{n + 1}$
  distinct for $i = 1, \ldots, r$ such that
  \begin{equation}
    \bar{\sigma} \equiv \sum_{i = 1}^r \lambda_i \mathbf{e}_{\bar{\zeta}_i}
  \ \mathrm{on}\ S_{[2 k]} \label{eq:hom} .
  \end{equation}
  Notice that for any $\lambda \neq 0$, $\mathbf{e}_{\bar{\zeta}_i} =
  \lambda^{- k} \mathbf{e}_{\lambda \bar{\zeta}_i}$ on $S_{[2 k]}$.

  By an inverse change of coordinates, the points $\bar{\zeta}_i$ of
  (\ref{eq:hom}) are transformed into some points $\bar{\zeta}_i = (\zeta_{i,
  0}, \zeta_{i, 1}, \ldots, \zeta_{i, n}) \in \mathbb{K}^{n + 1}$ such that
  $\zeta_{i, 0} \neq 0$ (say for $i = 1, \ldots, r'$) and the remaining $r -
  r'$ points with $\zeta_{i, 0} = 0$. \ The image by $\tau_0^{\ast}$ of
  $\mathbf{e}_{\bar{\zeta}_i} \in S_{[2 k]}^{\ast}$ with $\zeta_{i, 0} \neq
  0$ is
  \[ \tau_0^{\ast} (\mathbf{e}_{\bar{\zeta}_i}) \equiv \zeta_{i, 0}^{2 k}
     \mathbf{e}_{\zeta_i}  \]
  where $\zeta_i = \frac{1}{\zeta_{i, 0}} (\zeta_{i, 1}, \ldots, \zeta_{i,
  n})$. The image by $\tau_0^{\ast}$ of $\mathbf{e}_{\bar{\zeta}_i} \in
  S_{[2 k]}^{\ast}$ with $\zeta_{i, 0} = 0$ is a linear form \ $\in R_{2
  k}^{\ast}$, which vanishes on all the monomials $\mathbf{x}^{\alpha}$ with
  $| \alpha | < 2 k$, since their homogenization in degree $2 k$ is $x_0^{2 k
  - | \alpha | } \mathbf{x}^{\alpha}$ and their evaluation at $\bar{\zeta}_i
  = (0, \zeta_{i, 1}, \ldots, \zeta_{i, n})$ gives $0$. The value of
  $\tau_0^{\ast} (\mathbf{e}_{\bar{\zeta}_i})$ \ at $\mathbf{x}^{\alpha}$
  with $| \alpha | = 2 k$ \ $\zeta_{i, 1}^{\alpha_1} \cdots \zeta_{i,
  n}^{\alpha_n} =\mathbf{e}_{\zeta_i} (\mathbf{x}^{\alpha})$ where
  $\zeta_i = (\zeta_{i, 1}, \ldots, \zeta_{i, n})$. We deduce that
  \[ \tau_0^{\ast} (\mathbf{e}_{\bar{\zeta}_i}) \equiv
     \mathbf{e}_{\zeta_i} \circ \pi_{2 k} \]
  The flat extension theorem implies that if $H_{\sigma} \succcurlyeq 0$ then
  $\lambda_i > 0$ and $\bar{\zeta}_i \in \mathbb{R}^{n + 1}$ in the
  decomposition (\ref{eq:hom}). By dehomogenization, we have $\omega_i =
  \lambda_i \zeta_{i, 0}^{2 k} > 0$, \ $\zeta_i = \frac{1}{\zeta_{i, 0}}
  (\zeta_{i, 1}, \ldots, \zeta_{i, n}) \in \mathbb{R}^n$ for $i = 1, \ldots,
  r'$ and $\zeta_i = (\zeta_{i, 1}, \ldots, \zeta_{i, n}) \in \mathbb{R}^n$
  for $i = r' + 1, \ldots, n$.
\end{proof}

We exploit this structure theorem to show that if the truncation order is sufficiently high, a minimizer of (\ref{eq:min:norm}) corresponds to a cubature formula.

\begin{theorem}
  Let $P$ be a proper operator and $k \geqslant \frac{\deg (V) + 1}{2}$.
  Assume that there exists $\sigma^{\ast} \in R_{2 k}^{\ast}$ be such that
  $H_{\sigma^{\ast}}$ is a minimizer of (\ref{eq:min:norm}) of rank $r$ with
  $r \leqslant k$. Then $H_{\sigma^{\ast}} \in \mathcal{E}^k_r (\mathbf{i})$
  i.e. there exists $\omega_i > 0$ and $\zeta_i \in \mathbb{R}^n$ such that
  \[ \sigma^{\ast} \equiv \sum_{i = 1}^r \omega_i \mathbf{e}_{\zeta_i} . \]
\end{theorem}

\begin{proof}
  By Theorem \ref{thm:flat:trunc},
  \[ \sigma^{\ast} \equiv \sum_{i = 1}^{r'} \omega_i \mathbf{e}_{\zeta_i} +
     \sum_{i = r' + 1}^r w_i \mathbf{e}_{\zeta_i} \circ \pi_{2 k}, \]
  with \ $\omega_i > 0$ and $\zeta_i \in \mathbb{R}^n$ for $i = 1, \ldots,
  r$.

  Let us suppose that $r \neq r' .$ As $k \geqslant \frac{\deg (V) + 1}{2}$,
  the elements of $V$ are of degree $< 2 k$, therefore $\sigma^{\ast}$ and
  $\sigma' \equiv \sum_{i = 1}^{r'} \omega_i \mathbf{e}_{\zeta_i} $ coincide
  on $V$ and $H_{\sigma'} \in \mathcal{H}^k (\mathbf{i})$. We have the
  decomposition
  \[ \trace (P H_{\sigma^{\ast}} P) = \langle \sigma^{\ast} \mid q
     \rangle = \langle \sigma' \mid q \rangle_+  \sum_{i = 1}^{r'} \omega_i
     \pi_{2 k} (q) (\zeta_i) . \]
  The homogeneous component of highest degree $\pi_{2 k} (q)$ of $q
  (\mathbf{x}) = \sum_{i = 1}^{s_k} p_i (\mathbf{x})^2$ is the sum of the
  squares of the degree-$k$ components of the $p_i$:
  \[ \pi_{2 k} (q) = \sum_{i = 1}^{s_k}  (\pi_k (p_i^{}))^2, \]
  so that $\sum_{i = 1}^{r'} \omega_i \pi_{2 k} (q) (\zeta_i) \geqslant 0$. As
  $\trace (P H_{\sigma^{\ast}} P)$ is minimal, we must have $ \sum_{i =
  r' + 1}^r \omega_i \pi_{2 k} (q) (\zeta_i) = 0$, which implies that $\pi_{2
  k} (q) (\zeta_i)$ for $i = r' + 1, \ldots, r$. However, this is impossible, since
  $P$ is proper. We thus deduce that $r' = r$, which concludes the proof of the
  theorem.
\end{proof}

This theorem shows that an optimal solution of the minimization problem
(\ref{eq:min:norm}) of small rank $(r \leqslant k)$ yields a cubature formula,
which is exact on $V$. Among such minimizers, we have those of minimal
rank as shown in the next proposition.

\begin{proposition}
  Let $k \geqslant \frac{\deg (V) + 1}{2}$ and $H$ be an element of
  $\mathcal{H}^k (\mathbf{i})$ with minimal rank $r$. If $k \geqslant r$,
  then $H \in \mathcal{E}^k_r (\mathbf{i})$ and it is either an extremal
  point of $\mathcal{H}^k (\mathbf{i})$ or on a face of $\mathcal{H}^k
  (\mathbf{i})$, which is included in $\mathcal{E}^k_r (\mathbf{i})$.
\end{proposition}

\begin{proof}
  Let $H_{\sigma} \in \mathcal{E}^k_r (\mathbf{i})$ be of minimal rank $r$.

  By Theorem (\ref{thm:flat:trunc}), $\sigma \equiv \sum_{i = 1}^{r'} \omega_i
  \mathbf{e}_{\zeta_i} + \sum_{i = r' + 1}^r w_i \mathbf{e}_{\zeta_i}
  \circ \pi_{2 k}$ with $\omega_i > 0$ and $\zeta_i \in \mathbb{R}^n$ for $i
  = 1, \ldots, r$. The elements of $V$ are of degree $< 2 k$, therefore
  $\sigma$ and \ $\sigma' \equiv \sum_{i = 1}^{r'} \omega_i
  \mathbf{e}_{\zeta_i} $ coincide on $V$. We deduce that $H_{\sigma'} \in
  \mathcal{H}^k (\mathbf{i})$.

  As $\rank H_{\sigma'} = r' \leqslant r$ and $H_{\sigma} \in \mathcal{H}^k
  (\mathbf{i})$ is of minimal rank $r$, $r = r'$ and \ $H_{\sigma} \in
  \mathcal{E}^k_r (\mathbf{i})$.

  Let us assume that $H_{\sigma}$ is not an extremal point of $\mathcal{H}^k
  (\mathbf{i})$. Then it is in the relative interior of a face $F$ of
  $\mathcal{H}^k (\mathbf{i})$. For any $H_{\sigma_1}$ in a sufficiently
  small ball of $F$ around $H_{\sigma}$, there exist $t \in \nobracket ] 0, 1
  [\nobracket$ and $H_{\sigma_2} \in F$ such that
  \[ H_{\sigma} = t H_{\sigma_1} \nocomma \noplus + (1 - t) H_{\sigma_2} . \]
  The kernel of $H_{\sigma}$ is the set of polynomials $p \in R_k$ such that
  \[ 0 = H_{\sigma} (p, p) = t H_{\sigma_1} \nocomma \noplus (p, p) + (1 - t)
     H_{\sigma_2} (p, p) . \]
  As $H_{\sigma_i} \nocomma \noplus \succcurlyeq 0$, we have $H_{\sigma_i}
  \nocomma \noplus (p, p) = 0$ for $i = 1, 2$. This implies that $\ker
  H_{\sigma} \subset \ker H_{\sigma_i}$, for $i = 1, 2$. From the inclusion
  $\ker H_{\sigma_1} \cap \ker H_{\sigma_2} \subset \ker H_{\sigma}$, we
  deduce that
  \[ \ker H_{\sigma} = \ker H_{\sigma_1} \cap \ker H_{\sigma_2} . \]
  As $H_{\sigma}$ is of minimal rank $r$, we have \ $\dim \ker H_{\sigma}
  \geqslant \dim \ker H_{\sigma_i}$. This implies that $\ker H_{\sigma} = \ker
  H_{\sigma_1} = \ker H_{\sigma_2} .$

  As $r \leqslant k$, $\mathcal{A}_{\sigma}$ has a monomial basis $B$
  (connected to $1$) in degree $< k$ and $R_k = \langle B \rangle \oplus \ker
  H_{\sigma}$. Consequently, $H_{\sigma}$ (resp. $H_{\sigma_i}$) is a flat
  extension of $H_{\sigma}^{B, B}$ (resp. $H_{\sigma_i}^{B, B}$) and we have
  the decomposition
  \[ \sigma_1 \equiv \sum_{i = 1}^r \omega_{i, 1}^{} \mathbf{e}_{\zeta_i}
     \nocomma, \hspace{1em} \sigma_2 \equiv \sum_{i = 1}^r \omega_{i, 2}^{}
     \mathbf{e}_{\zeta_i}, \]
  with $\omega_{i^{}, j} > 0$, $i = 1, \ldots, r$, $j = 1, 2$. We deduce that
  $H_{\sigma_1} \in \mathcal{E}^k_r (\mathbf{i})$ and all the elements of
  the line $(H_{\sigma}, H_{\sigma_1})$ which are in $F$ are also in
  $\mathcal{E}^k_r (\mathbf{i})$. Since $F$ is convex, we deduce that $F
  \subset \mathcal{E}^k_r (\mathbf{i})$.
\end{proof}

\begin{remark}
A cubature formula is {{\em interpolatory}} when the weights are
uniquely determined from the points. From the the previous Theorem and Proposition,
we see that if a cubature formula is of minimal rank and
interpolatory, then it is an extremal point of $\mathcal{H}^k (\mathbf{i})$.
\end{remark}
According to the previous proposition, by minimizing the nuclear norm of a random matrix,
we expect to find an element of minimal rank in one of the faces of $\mathcal{H}^k
(\mathbf{i})$, provided $k$ is big enough. This yields the following
simple algorithm, which solves the SDP problem, and checks the flat extension
property using Algorithm \ref{algo:1}. Furthermore it  computes the decomposition using
Algorithm \ref{algo:dec} or increases the degree if there is no flat extension:

{\algorithm{\begin{itemize}
  \label{algo:2}\item $k \assign \left\lceil \frac{\deg (V)}{2} \right\rceil
  ;$ notflat := true; P:= random $s_k \times s_k $ matrix;

  \item While{\textbf{}} (notflat) do
  \begin{itemize}
    \item Let $\sigma$ be a solution of the SDP problem: $\min_{H \in
    \mathcal{H}^k (\mathbf{i})} \trace (P^t H P) ;$

    \item If $H_{\sigma}^k$ is not a flat extension, then $k \assign k + 1$;
    else notflat:=false;
  \end{itemize}
  \item Compute the decomposition of $\sigma = \sum_{i = 1}^r w_i
  \hspace{0.17em} \mathbf{e}_{\zeta_i}, \omega_i > 0, \zeta_i \in
  \mathbb{R}^n$.
\end{itemize}}}

\section{Examples}

We now illustrate our cubature method on a few explicit examples.

\begin{example}[Cubature on a square] Our first application is a well known case, namely, the square domain $\Omega = [- 1, 1] \times [- 1, 1]$.
  We solve the SDP problem \ref{eq:min:norm}, with a random matrix $P$ and with no
  constraint on the support of the points. In following table, we give the
  degree of the cubature formal (i.e. the degree of the polynomials for which
  the cubature formula is exact), the number $N$ of cubature points, the
  coordinates of the cubature points and the associated weights.

\begin{center}\small
    \begin{tabular}{|c|c|c|c|}
      \hline
      Degree & N & Points & Weights\\
      \hline
      3 & 4 & $\pm$(0.46503, 0.464462) & 1.545\\
      &  & $\pm$(0.855875, -0.855943) & 0.454996\\
      \hline
      5 & 7 & $\pm$(0.673625, 0.692362) & 0.595115\\
      &  & $\pm$(0.40546, -0.878538) & 0.43343\\
      &  & $\pm$(-0.901706, 0.340618) & 0.3993\\
      &  & (0, 0) & 1.14305\\
      \hline
      7 & 12 & $\pm$(0.757951, 0.778815) & 0.304141\\
      &  & $\pm$(0.902107, 0.0795967) & 0.203806\\
      &  & $\pm$(0.04182, 0.9432) & 0.194607\\
      &  & $\pm$(0.36885, 0.19394) & 0.756312\\
      &  & $\pm$(0.875533, -0.873448) & 0.0363\\
      &  & $\pm$(0.589325, -0.54688) & 0.50478\\
      \hline
    \end{tabular}
\end{center}
\end{example}
The cubature points are symmetric with respect to the origin $(0, 0)$. The
computed cubature formula involves the minimal number of points, which all lie
in the domain $\Omega$.

\begin{example}[Barycentric Wachspress coordinates on a pentagon]
 Here we consider the pentagon $C$ of vertices $v_1 = (0, 1), v_2 = (1, 0), v_3 =
  (- 1, 0), v_4 = (- 0.5, - 1), v_5 = (0.5, - 1)$.
  \begin{center}
    \resizebox{9cm}{!}{\includegraphics{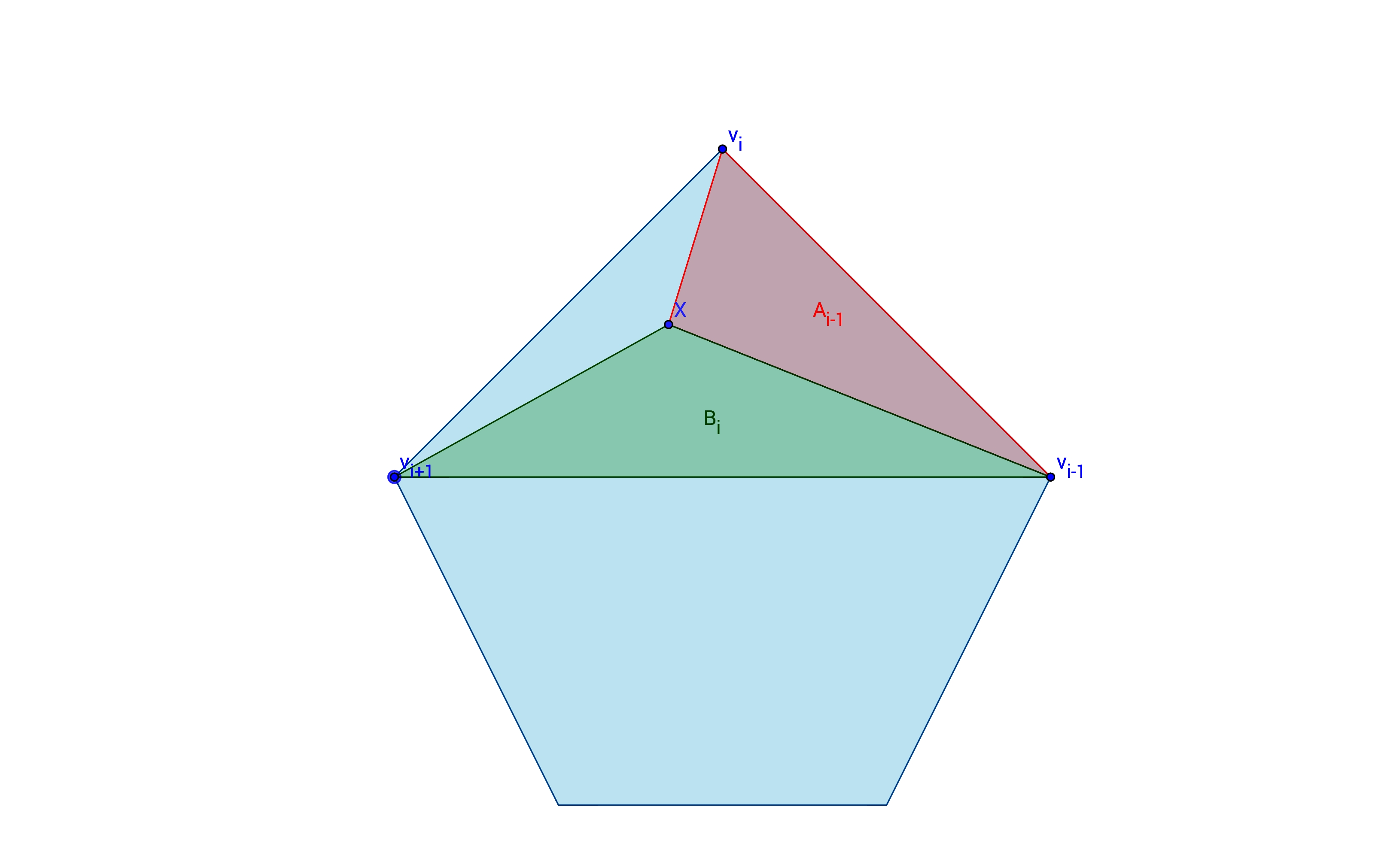}}
  \end{center}
  To this pentagon, we associate (Wachpress) barycentric coordinates \cite{Wachpress75}, which
  are defined as follows. The weighted function associated to the vertex $v_i$
  is defined as:
  \[ w_i (\xb) = \frac{A_{i - 1} - B_i + A_i}{A_{i - 1} \cdot A_i} \]
  where $A_i$ is the signed area of the triangle $(\mathbf{x}, v_{i - 1},
  v_i)$ and $B_i$ is the area of $(\mathbf{x}, v_{i + 1}, v_{i - 1})$. The
  coordinate function associated to $v_i$ is:
  \[ \lambda_i (\xb) = \frac{w_i (\xb)}{\sum_{i = 1}^5 w_i (\xb)} . \]
  These coordinate functions satisfy:
  \begin{itemize}
    \item $\lambda_i (\xb) \ge 0$ for $\mathbf{x} \in C$

    \item $\sum_{i = 1}^5 \lambda_i = 1,$

    \item $\sum_{i = 1}^5 v_i \cdot \lambda_i (\xb) = \xb$
  \end{itemize}
  For all polynomials $p \in R =\mathbb{R} [u_0, u_1, u_2, u_3, u_4]$, we
  consider
  \[ I [p] = \int_{\xb \in \Omega} p \circ \lambda (\xb) d \xb \]
  We look for a cubature formula $\sigma \in R^{\ast}$ of the form:
  \begin{equation}
    \langle \sigma \mid p \rangle = \sum_{j = 1}^r w_j p (\zeta_j)
  \end{equation}
  with $w_i > 0$, $\zeta_i \in \mathbb{R}^5$, such that $I [p] = \langle
  \sigma \mid p \rangle$ for all polynomial $p$ of degree $\le 2$.

  The moment matrix $H_{\sigma}^{B, B}$ associated to $B = \{1, u_0, u_1, u_2,
  u_3, u_4 \}$ involves moments of degree $\leqslant 2$:
  {\small
    \[ H_{\sigma}^{B, B} = \left( \begin{array}{cccccc}
       2.5000 & 0.5167 & 0.5666 & 0.5167 & 0.4500 & 0.4500\\
       0.5167 & 0.2108 & 0.1165 & 0.0461 & 0.0440 & 0.0992\\
       0.5666 & 0.1165 & 0.2427 & 0.1165 & 0.0454 & 0.0454\\
       0.5167 & 0.0461 & 0.1165 & 0.2108 & 0.0992 & 0.0440\\
       0.4500 & 0.0440 & 0.0454 & 0.0992 & 0.1701 & 0.0911\\
       0.4500 & 0.0992 & 0.0454 & 0.0440 & 0.0911 & 0.1701
  \end{array} \right) \]
 }
  Its rank is $\rank (H^{B, B}_{\sigma}) = 5$.

  We compute $H^{B^+, B^+}_{\sigma}$. In this matrix there are 105 unknown
  parameters. We solve the following SDP problem
  \begin{eqnarray}
    \min &  & \trace (H_{\sigma}^{B^+, B^+}) \\
    s.t. &  & H^{B^+, B^+}_{\sigma} \succcurlyeq 0 \nonumber
  \end{eqnarray}
  which yields a solution with minimal rank 5. Since the rank of the solution
  matrix is the rank of $H_{\sigma}^{B, B}$, we do have a flat extension.
  Applying Algorithm \ref{algo:1}, we find the orthogonal polynomials
  $\rho_i$, the matrices of the operators of multiplication by a variable,
  their common eigenvectors, which gives the following cubature points and
  weights:
  {\small
  \[ \begin{array}{|l|l|}\hline
       \mathrm{Points} & \mathrm{Weights}\\\hline
       (0.249888, - 0.20028, 0.249993, 0.350146, 0.350193) & 0.485759\\
       (0.376647, 0.277438, - 0.186609, 0.20327, 0.329016) & 0.498813\\
       (0.348358, 0.379898, 0.244967, - 0.174627, 0.201363) & 0.509684\\
       (- 0.18472, 0.277593, 0.376188, 0.329316, 0.201622) & 0.490663\\
       (0.242468, 0.379314, 0.348244, 0.200593, - 0.170579) & 0.51508\\ \hline
     \end{array} \]
   }
 \end{example}

\bibliographystyle{plain}

\end{document}